\documentclass{amsart}
\usepackage{amssymb}
\usepackage{hyperref}

\usepackage{color,soul}
\definecolor{ItalianApricot}{rgb}{1,0.7,0.5}
\sethlcolor{ItalianApricot}

\theoremstyle{plain}
\newtheorem{thm}{Theorem}[section]
\newtheorem{prop}[thm]{Proposition}
\newtheorem{lem}[thm]{Lemma}
\newtheorem{cor}[thm]{Corollary}

\theoremstyle{definition}
\newtheorem{defn}[thm]{Definition}

\theoremstyle{remark}
\newtheorem{rem}[thm]{Remark}

\numberwithin{equation}{section}

\renewcommand{\epsilon}{\varepsilon}
\renewcommand{\phi}{\varphi}

\DeclareMathOperator{\uh}{\upharpoonright}

\newcommand{\R}{\mathbb{R}}

\DeclareMathOperator{\cdim}{cdim}

\DeclareMathOperator{\KM}{\it KM}

\newcommand{\Pf}[1][]{\operatorname{P\!}_{f_{#1}}\!} % f-potential
\newcommand{\Ef}{\operatorname{E}_f} % f-energy
\newcommand{\Cf}[1][]{\operatorname{C}_{f_{#1}}} % f-capacity

\newcommand{\RP}{\operatorname{\widetilde{P}\!}_s\!} % Riesz s-potential
\newcommand{\RE}{\operatorname{\widetilde{E}}_s}	% Riesz s-energy

\newcommand{\Ps}{\operatorname{P\!}_s\!} % s-potential
\newcommand{\Es}{\operatorname{E}_s} % s-energy
\newcommand{\Cs}{\operatorname{C}_s} % s-capacity

\newcommand{\RElog}{\operatorname{\widetilde{E}}_{\log}}	% Riesz log-energy

\newcommand{\RPklog}{\operatorname{\widetilde{P}\!}_{\log^k}\!} % Riesz log-potential
\newcommand{\REklog}{\operatorname{\widetilde{E}}_{\log^k}}	% Riesz log-energy

\newcommand{\ww}{\operatorname{w}} % dynamic weight

\begin{document}

\title{Energy randomness}

%\date{June 17, 2014}
\date{\today}

\author[Miller]{Joseph S.~Miller}
\address[Miller]{Department of Mathematics\\
University of Wisconsin\\
Madison, WI 53706, USA}
\email{jmiller@math.wisc.edu}

\author[Rute]{Jason Rute}
\address[Rute]{Department of Mathematics\\
Pennsylvania State University\\
University Park, PA 16802, USA}
\email{jmr71@math.psu.edu}

\thanks{This work was started while the authors were participating in the Program on Algorithmic Randomness at the Institute for Mathematical Sciences of the National University of Singapore in June 2014. The authors would like to thank the institute for its support. Miller was partially supported by NSF grant DMS-1001847.}


% 03-XX: Mathematical logic and foundations
%	03Dxx: Computability and recursion theory
%		03D32: Algorithmic randomness and dimension [See also 68Q30]
%
% 68-XX: Computer science For papers involving machine computations and
%       programs in a specific mathematical area, see Section--04 in that area
%	68Qxx: Theory of computing
%		68Q30: Algorithmic information theory (Kolmogorov complexity, etc.)
%       [See also 03D32]
%
% 31-XX Potential theory {For probabilistic potential theory, see 60J45}
%	31Cxx 		Other generalizations
%		31C15   	Potentials and capacities

\begin{abstract}
Energy randomness is a notion of partial randomness introduced by Diamondstone and Kjos-Hanssen to characterize the sequences that can be elements of a Martin-L{\"o}f random closed set (in the sense of Barmpalias, Brodhead, Cenzer, Dashti, and Weber).  It has also been applied by Allen, Bienvenu, and Slaman to the characterization of the possible zero times of a Martin-L{\"o}f random Brownian motion.  In this paper, we show that $X \in 2^\omega$ is $s$-energy random if and only if $\sum_{n\in\omega} 2^{sn - \KM(X\uh n)} < \infty$, providing a characterization of energy randomness via a priori complexity $\KM$.  This is related to a question of Allen, Bienvenu, and Slaman.
\end{abstract}

\maketitle

%%%%%%%%
%%%%%%%%
\section{Introduction}
%%%%%%%%
%%%%%%%%

Algorithmic randomness is a branch of computability theory that studies objects that behave randomly with respect to computable statistical tests.   The most common randomness notion, Martin-L{\"o}f randomness, was first used to study random sequences in the space $2^\omega$ with respect to the fair-coin measure.  Since then, Martin-L{\"o}f randomness has been extended to other measures on $2^\omega$, including noncomputable measures.  It has also been extended to other spaces of objects.  For example, Fouch{\'e} \cite{Fouche2000}---building on work of Asarin and Pokrovskii \cite{Asarin.Pokrovskii:1986}---studied Martin-L{\"o}f random Brownian motion.  Later, Barmpalias, Brodhead, Cenzer, Dashti, and Weber \cite{Barmpalias2007} introduced a particular notion of Martin-L{\"o}f randomness for closed subsets of $2^\omega$.  

In an effort to characterize the sequences that are possible elements of Martin-L{\"o}f random closed sets, Diamondstone and Kjos-Hanssen \cite{DK:12} introduced $s$-energy randomness. (The concept was also implicitly mentioned in work of Reimann \cite{Reimann:2008vn}.)  A sequence $X\in 2^\omega$ is $s$-energy random (where $s$ is computable and $0<s<1$) if and only if $X$ is Martin-L{\"o}f random for some (not necessarily computable) measure $\mu$ on $2^\omega$ such that $\mu$ has finite Riesz $s$-energy,
\[
\iint \rho(X,Y)^{-s} \;d\mu(Y)d\mu(X).
\]
(Here $\rho$ is the standard metric on $2^\omega$.) The notion of energy comes from potential theory, and there is a strong connection between potential theory and probability theory.    Diamondstone and Kjos-Hanssen showed that if $X$ is $\log_2 (3/2)$-energy random, then $X$ is the member of some Martin-L{\"o}f random closed set.  (The converse direction will be proved in an upcoming paper by the second author \cite{Rute:aa}.)

Diamondstone and Kjos-Hanssen also showed a close relationship between energy randomness and effective Hausdorff dimension.  Namely, the constructive dimension of $X$ can be characterized via
\[
\cdim{X} = \sup\{s\colon X \text{ is } s\text{-energy random}\},
\]
where the supremum is over computable $s\in(0,1)$.

Allen, Bienvenu, and Slaman \cite{Allen:aa} studied a similar problem to that of Diamondstone and Kjos-Hanssen, namely the classification of zero times of a Martin-L{\"o}f random Brownian motion. They showed that for $t>0$, if $\cdim(t)>1/2$, then $B(t)=0$ for some Martin-L{\"o}f random Brownian motion $B$, and if $\cdim(t)<1/2$, then $B(t)\neq 0$ for all Martin-L{\"o}f random Brownian motions $B$.  While their work does not explicitly mention energy randomness, it does rely on calculations involving $1/2$-energy, which suggests a connection.  (The exact correspondence between $1/2$-energy randomness and the zero times of a Martin-L{\"o}f random Brownian motion---as well as other applications of energy randomness to multidimensional Brownian motion---will be addressed in \cite{Rute:aa}.)  Allen, Bienvenu, and Slaman \cite{Allen:aa} also asked whether the zero times of Martin-L{\"o}f random Brownian motion can be characterized via complexity.

The goal of this paper is to characterize $s$-energy randomness via a priori complexity $\KM$.
\begin{thm}\label{thm:main-s-energy} 
Let $s \in (0,1)$ be computable. A sequence $X \in 2^\omega$ is $s$-energy random if and only if
\begin{equation}
\sum_{n\in\omega} 2^{sn - \KM(X\uh n)} < \infty.\label{eq:main-result}
\end{equation}
\end{thm}

In order to prove Theorem~\ref{thm:main-s-energy}, we will prove a stronger result.  We generalize $s$-energy randomness to $f$-energy randomness for functions $f\colon\omega\rightarrow[0,\infty)$.  In particular, $f(n)=2^{sn}$ will correspond to $s$-energy randomness.  In Theorem~\ref{thm:main}, we show that for certain functions $f$, $X$ is $f$-energy random if and only if
\[ 
\sum_{n\in\omega} f(n)2^{-\KM(X\uh n)} < \infty.
\]

A direct application of Theorem~\ref{thm:main-s-energy}, when combined with the aforementioned results of Kjos-Hanssen and Diamondstone, is that if $X$ satisfies (\ref{eq:main-result}) with $s=\log_2 (3/2)$, then $X$ is the member of some Martin-L{\"o}f random closed set.  Besides such applications to random closed sets and random Brownian motion, Theorem~\ref{thm:main-s-energy} is interesting for the follow three reasons.

Theorem~\ref{thm:main-s-energy} is very similar in form to the Ample Excess Lemma (proved by Miller and Yu~\cite{Miller.Yu:ta1} but implicit in G{\'a}cs~\cite[Proof of Theorem~5.2]{G:80}).  This result says the $X \in 2^\omega$ is Martin-L{\"o}f random if and only if 
\[ 
\sum_{n\in\omega} 2^{n-K(X\uh n)} < \infty,
\]
where $K$ is prefix-free Kolmogorov complexity.

Theorem~\ref{thm:main-s-energy} is also very similar to the Effective Capacitability Theorem of Reimann \cite[Theorem~14]{Reimann:2008vn}, which states that the following are equivalent for a computable real $s\in(0,1)$.
\begin{enumerate}
\item $X\in 2^\omega$ is Martin-L{\"o}f random for some probability measure $\mu$ on $2^\omega$ such that there exists some $C>0$ such that for all $\sigma \in 2^{<\omega}$, $\mu[\sigma] \leq C 2^{-s |\sigma|}$.
\item $X$ is strongly $s$-random, that is, for all $n$, $\KM(X\uh n) \geq sn + O(1)$.
\end{enumerate}

Last, Theorem~\ref{thm:main-s-energy} hints at a deep connection between potential theory and algorithmic information theory.  Potential theory concerns three basic notions: potential, energy, and capacity.  In order to prove Theorem~\ref{thm:main-s-energy}, we will use all three extensively. In particular, in this paper, the $s$-potential of a measure $\mu$ is the function
\[
\Ps\mu(X) = \sum_{n\in\omega} 2^{-sn}\mu[X\uh n].
\]
The function $\mathbf{M}(\sigma) = 2^{-\KM(\sigma)}$ is known as the universal left-c.e.\ semimeasure.  Theorem~\ref{thm:main-s-energy} can be interpreted as saying that $X$ is $s$-energy random if and only if the $s$-potential of the universal left-c.e.\ semimeasure $\mathbf{M}$ evaluated at $X$ is finite.  It would be interesting to understand what it means to take the potential of a semimeasure.

%%%%%%%%
%%%%%%%%
\section{Effective randomness}
%%%%%%%%
%%%%%%%%

Let $2^{<\omega}$ denote the set of finite binary strings, and let $2^\omega$ denote the space of infinite binary sequences.  We write $\lambda$ for the string of zero length.  For $\sigma \in 2^{<\omega}$, let $[\sigma]$ denotes the cylinder set of all $X \in 2^\omega$ that extend $\sigma$.  If $U \subseteq 2^{<\omega}$, let $[U]=\bigcup \{[\sigma] \colon \sigma \in U\}$. 

When we say that $\mu$ is a \emph{measure}, we mean $\mu$ is a finite Borel measure on $2^\omega$.  In particular, every nonnegative function $\mu$ on cylinder sets that satisfies $\mu[\sigma 0]+\mu[\sigma 1] = \mu[\sigma]$ can be uniquely extended to a measure.  A probability measure is a measure $\mu$ such that $\mu[\lambda]=1$.

A semimeasure $\rho$ is a nonnegative function $\rho$ on $2^{<\omega}$ that satisfies $\rho(\sigma 0) + \rho(\sigma 1) \leq \rho(\sigma)$ and such that $\rho(\lambda)<1$.  There is an effective enumeration of the left-c.e.\ semimeasures $\rho_i$ and therefore a universal left-c.e.\ semimeasure $\mathbf{M} = \sum_{i\in\omega} 2^{-i}\rho_i$ which mulitplicatively dominates all other left-c.e.\ semimeasures.  \emph{A priori complexity} is defined as $\KM(\sigma) = -\log_2 \mathbf{M}(\sigma)$.  

It is well-known that notions from computability theory, such as a priori complexity and $\Sigma^0_1$-classes, can be relativized to an oracle $A \in 2^\omega$, giving us $\KM^A$ and $\Sigma^0_1[A]$.  The same can be done if the oracle is a measure $\mu$, even though there may not be a minimal Turing degree that computes $\mu$.  For example, to define $\KM^\mu$, note that a semimeasure $\rho$ is left-c.e.\ in $\mu$ if and only if $\rho$ is left-c.e.\ in $A$ for every $A \in 2^\omega$ that computes $\mu$.  It is possible to effectively enumerate all left-c.e.\ semimeasures $\{\rho_i\}_{i\in \omega}$ computable in $\mu$ and thus to get a universal $\mu$-left-c.e.\ semimeasure $\mathbf{M}^\mu$.  As before, we let $\KM^\mu = -\log_2 \mathbf{M}^\mu$.  Note that $\KM^\mu(\sigma) \leq \KM(\sigma)$ up to a multiplicative constant. For more on using measures as oracles, see Reimann and Slaman~\cite{RS:15} and Day and Miller~\cite{Day:2013aa}.

\begin{defn}
Let $\mu$ be a measure (not necessarily computable). A computable sequence $\{U_n\}_{n\in\omega}$ of $\Sigma^0_1[\mu]$-classes is a \emph{$\mu$-test} if $(\forall n)\; \mu(U_n)\leq 2^{-n}$. A sequence $X\in 2^\omega$ \emph{passes} the $\mu$-test if $X\notin\bigcup_{n\in\omega} U_n$. We say that $X\in 2^\omega$ is \emph{$\mu$-random} if it passes every $\mu$-test.
\end{defn}

We will need a characterization of $\mu$-randomness using a priori complexity. It is a straightforward generalization of the well-known Lebesgue measure case.
\begin{prop}
A sequence $X \in 2^\omega$ is $\mu$-random if and only if there is a constant $C$ such that for all $n\in \omega$,
\[
\KM^\mu(X \uh n) \geq -\log_2 \mu[X \uh n] - C.
\]
\end{prop}
\begin{proof}
Assume that $X$ is not $\mu$-random.  Then $X$ fails some $\mu$-test $\{U_k\}_{k\in\omega}$.  Let $\nu$ be the measure given by $\nu[\sigma] = \int_{[\sigma]} \sum_{k\in \omega} \chi_{U_k}\;d\mu$.  This measure is left-c.e.\ in $\mu$ and $\nu[\lambda]= \sum_{k\in \omega} \mu(U_k) \leq 2$.  Therefore, there is a large enough natural number $i > 0$ such that $\mathbf{M}^\mu \geq 2^{-i} \nu$.  Let $k = 2^{j + i}$ for an arbitrary $j\in\omega$.  Since $X \in U_0 \cap \ldots \cap U_k$, there is an $n$ large enough that $[X\uh n] \subseteq U_0 \cap \ldots \cap U_k$.  Therefore $\mathbf{M}^\mu(X \uh n) \geq 2^{-i} \nu[X \uh n] \geq k 2^{-i} \mu[X \uh n] = 2^j \mu[X \uh n]$. Hence $\KM^\mu(X \uh n) \leq -\log_2 \mu[X\uh n] - j$ for an arbitrary large $j$.

Conversely, assume that $\KM^\mu(X \uh n) \leq -\log_2 \mu[X\uh n] - k$ for all $k$.  Let 
\[
U_k = \left\{Y \in 2^\omega \colon (\exists n)\; \frac{\mathbf{M}^\mu(Y \uh n)}{\mu[Y \uh n] } > 2^k \right\}.
\]
This is a computable sequence of $\Sigma^0_1[\mu]$-classes.  Moreover, let $S_k \subseteq 2^{<\omega}$ be a prefix-free set such that $U_k = [S_k]$ and such that for all $\sigma \in S_k$, $\mathbf{M}^\mu(\sigma) \geq 2^k \mu[\sigma]$.  Then 
\begin{align*}
\mu(U_k) 
= \sum_{\sigma \in S_k} \mu[\sigma] 
\leq 2^{-k} \sum_{\sigma \in S_k} \mathbf{M}^\mu(\sigma) 
\leq 2^{-k}
\end{align*}
Therefore $\{U_k\}_{k\in\omega}$ is a $\mu$-test.  Since $X \in \bigcap_{k\in \omega} U_k$, $X$ fails the test and is not $\mu$-random.
\end{proof}

%%%%%%%%
%%%%%%%%
\section{Energy randomness}
%%%%%%%%
%%%%%%%%

Fix a function $f\colon\omega\to[0,\infty)$. Before we define what it means for a sequence $X\in 2^\omega$ to be $f$-energy random, we must define $f$-potential and $f$-energy.  (While the presentation in this paper is self-contained, the reader may consult the book of Lyons and Peres \cite{Lyons:aa} for more on potential and energy, including a physical interpretation where $f$ is the resistance in an electrical network.)

\begin{defn}
Let $\mu$ be a measure on $2^\omega$. The \emph{$f$-potential} of $\mu$ is the function $\Pf\mu(X) = \sum_{n\in\omega} f(n)\mu[X\uh n]$. The \emph{$f$-energy} of $\mu$ is
\[
\Ef(\mu) = \int \Pf\mu(X) \;d\mu(X) = \int \sum_{n\in\omega} f(n)\mu[X\uh n] \;d\mu(X).
\]
\end{defn}

\begin{lem}\label{lem:two-measures}
Let $\mu$ and $\nu$ be any two measures. Then
\[
\int \Pf\nu(X) \;d\mu(X) = \sum_{\sigma\in 2^{<\omega}} f(|\sigma|)\mu[\sigma]\nu[\sigma].
\]
\end{lem}
\begin{proof}
\begin{align*}
\int \Pf\nu(X) \;d\mu(X) &= \int \sum_{n\in\omega} f(n)\nu[X\uh n] \;d\mu(X) = \sum_{n\in\omega} \int f(n)\nu[X\uh n] \;d\mu(X) \\
	&= \sum_{n\in\omega}\sum_{\sigma\in 2^n} f(n)\nu[\sigma]\mu[\sigma] = \sum_{\sigma\in 2^{<\omega}} f(|\sigma|)\nu[\sigma]\mu[\sigma]. \qedhere
\end{align*}
\end{proof}

If we take $\nu=\mu$ in the previous lemma, we get a satisfying expression for the $f$-energy of a measure $\mu$.

\begin{prop}\label{prop:energy-characerization}
$\displaystyle\Ef(\mu) = \sum_{\sigma\in 2^{<\omega}} f(|\sigma|)\mu^2[\sigma]$.
\end{prop}

We are ready for the definition of $f$-energy randomness.

\begin{defn}
A sequence $X\in 2^\omega$ is \emph{$f$-energy random} if it is random for some $\mu$ such that $\Ef(\mu)<\infty$.
\end{defn}

\begin{prop}
Assume that $f\colon\omega\to[0,\infty)$ is computable. If $X$ is $\mu$-random and $\Ef(\mu)<\infty$, then
\begin{enumerate}
	\item\label{it:one} $\displaystyle\sum_{n\in\omega} f(n)\mu[X\uh n] < \infty$, and
	\item\label{it:two} $\displaystyle\sum_{n\in\omega} f(n)2^{-\KM(X\uh n)} < \infty$.
\end{enumerate}
\end{prop}
\begin{proof}
Fix $c$ such that
\[
\Ef(\mu) = \int \sum_{n\in\omega} f(n)\mu[X\uh n] \;d\mu(X) < c.
\]
For each $n$, consider the $\Sigma^0_1[\mu]$-class $U_n = \{X\colon \sum_{n\in\omega} f(n)\mu[X\uh n] > c2^n\}$. We have $\mu(U_n)\leq 2^{-n}$, so $\{U_n\}_{n\in\omega}$ is a $\mu$-test. So if $\sum_{n\in\omega} f(n)\mu[X\uh n] = \infty$, then $X$ is not random for $\mu$. This proves $\eqref{it:one}$.

If $X$ is random for $\mu$, then there is a $c$ such that $\KM(X\uh n)\geq -\log(\mu[X\uh n])-c$. In other words, $2^{-\KM(X\uh n)}\leq 2^c\mu[X\uh n]$. Combining this with $\eqref{it:one}$:
\[
\sum_{n\in\omega} f(n)2^{-\KM(X\uh n)} \leq 2^c\sum_{n\in\omega} f(n)\mu[X\uh n] < \infty. \qedhere
\]
\end{proof}

\begin{cor}\label{cor:one-direction}
If $f\colon\omega\to[0,\infty)$ is computable and $X\in 2^\omega$ is $f$-energy random, then $\displaystyle\sum_{n\in\omega} f(n)2^{-\KM(X\uh n)} < \infty$.
\end{cor}

Our main theorem is a converse to Corollary~\ref{cor:one-direction} under a fairly weak assumption on $f$. There are a couple of uninteresting cases that we can settle now. Let $|f| = \sum_{n\in\omega} f(n)2^{-n}$. The following two remarks tell us that it is safe to restrict our attention to the case where $|f|\in(0,\infty)$.

\begin{rem}\label{rem:too-fast}
Assume that $|f|=\infty$. If $\mu$ is a nonzero measure, then
\begin{align*}
\Ef(\mu) &= \sum_{\sigma\in 2^{<\omega}} f(|\sigma|)\mu^2[\sigma] = \sum_{n\in\omega} f(n)2^{n}\sum_{\sigma\in 2^n}\frac{\mu^2[\sigma]}{2^n}\geq \sum_{n\in\omega} f(n)2^{n}\left(\sum_{\sigma\in 2^n}\frac{\mu[\sigma]}{2^n}\right)^2 \\
	&= \sum_{n\in\omega} f(n)2^{n}\frac{\mu^2[\lambda]}{2^{2n}} = \mu^2[\lambda]\sum_{n\in\omega} f(n)2^{-n} = \mu^2[\lambda]|f| = \infty,
\end{align*}
where the inequality follows from the convexity of $x\mapsto x^2$ and Jensen's inequality. Since no nonzero measure has finite $f$-energy, no sequence is $f$-energy random.

On the other hand, recall that $\KM(\sigma)\leq |\sigma|$ up to an additive constant. Therefore, up to a multiplicative constant,
\[
\sum_{n\in\omega} f(n)2^{-\KM(X\uh n)} \geq \sum_{n\in\omega} f(n)2^{-n} = |f| = \infty,
\]
for any sequence $X\in 2^\omega$.

So in the case that $|f| = \infty$, we have that $X\in 2^\omega$ is $f$-energy random if and only if $\sum_{n\in\omega} f(n)2^{-\KM(X\uh n)} < \infty$, since both conditions are impossible.
\end{rem}

\begin{rem}\label{rem:fin-support}
Now let us assume that $f$ has finite support. This includes the case where $|f|=0$, i.e., where $f$ is the zero function. We claim $X\in 2^\omega$ is $f$-energy random if and only if $\sum_{n\in\omega} f(n)2^{-\KM(X\uh n)} < \infty$, since both conditions are always true. To see this, take any $X\in 2^\omega$. It is clear that $\sum_{n\in\omega} f(n)2^{-\KM(X\uh n)} < \infty$. For the other condition, let $\mu$ be a measure that has an atom at $X$. It is easy to see that $X$ is random for $\mu$ and $\Ef(\mu)<\infty$, so $X$ is $f$-energy random.
\end{rem}

%%%%%%%%
\subsection{The motivating example: \texorpdfstring{\boldmath$s$}{s}-energy randomness}
\label{subsec:s-energy}
%%%%%%%%

Energy randomness was introduced by Diamondstone and Kjos-Hanssen \cite{DK:12} in the form of $s$-energy randomness. Fix $s\in(0,1]$. Let $\rho$ denote the standard metric $\rho(X,Y)=\inf \{2^{-n} : X\uh n = Y \uh n \}$ on $2^\omega$.

\begin{defn}
Let $\mu$ be a measure on $2^\omega$. The \emph{Riesz $s$-potential} of $\mu$ is the function $\RP\mu(X) = \int \rho(X,Y)^{-s} \;d\mu(Y)$. The \emph{Riesz $s$-energy} of $\mu$ is
\[
\RE(\mu) = \int \RP\mu(X) \;d\mu(X) = \iint \rho(X,Y)^{-s} \;d\mu(Y)d\mu(X).
\]
\end{defn}

Define $X\in 2^\omega$ to be \emph{$s$-energy random} if it is random for some $\mu$ such that $\RE(\mu)<\infty$. We can fit $s$-energy randomness into the framework from the previous section by using slightly different, but essentially equivalent, notions of energy and potential.

\begin{defn}
Let $\mu$ be a measure on $2^\omega$. The \emph{$s$-potential} of $\mu$ is the function $\Ps\mu(X) = \sum_{n\in\omega} 2^{sn}\mu[X\uh n]$. The \emph{$s$-energy} of $\mu$ is
\[
\Es(\mu) = \int \Ps\mu(X) \;d\mu(X) = \int \sum_{n\in\omega} 2^{sn}\mu[X\uh n] \;d\mu(X).
\]
\end{defn}

In other words, the $s$-potential is the $f$-potential for $f(n)=2^{sn}$. There is a simple relationship between $s$-potential and Riesz $s$-potential.

\begin{lem}
$\RP\mu(X) = 2^{-s}\mu(2^\omega) + (1-2^{-s})\Ps\mu(X)$.
\end{lem}
\begin{proof}
We can rewrite $\rho(X,Y)^{-s}$ as
\[
2^{-s} + \sum_{n\in\omega} (2^{ns}-2^{(n-1)s})\chi_{[X\uh n]}(Y) = 2^{-s} + (1-2^{-s})\sum_{n\in\omega} 2^{ns}\chi_{[X\uh n]}(Y).
\]
Therefore,
\begin{align*}
\RP\mu(X) &= \int \rho(X,Y)^{-s} \;d\mu(Y) = \int 2^{-s} + (1-2^{-s})\sum_{n\in\omega} 2^{ns}\chi_{[X\uh n]}(Y) \;d\mu(Y) \\
	&= 2^{-s}\mu(2^\omega) + (1-2^{-s})\sum_{n\in\omega} 2^{ns}\int\chi_{[X\uh n]}(Y) \;d\mu(Y) \\
	&= 2^{-s}\mu(2^\omega) + (1-2^{-s})\sum_{n\in\omega} 2^{ns}\mu[X\uh n] = 2^{-s}\mu(2^\omega) + (1-2^{-s})\Ps\mu(X). \qedhere
\end{align*}
\end{proof}

It follows that:

\begin{prop}
$\RE(\mu) = 2^{-s}\mu^2(2^\omega) + (1-2^{-s})\Es(\mu)$. In particular, $\RE(\mu)<\infty$ if and only if $\Es(\mu)<\infty$.
\end{prop}

\begin{cor}\label{cor:s-energy-f-energy}
$X\in 2^\omega$ is $s$-energy random if and only if it is $f$-energy random for $f(n)=2^{sn}$.
\end{cor}

Note that we included the case when $s=1$, but if $f(n)=2^n$, then $|f| = \infty$. So by Remark~\ref{rem:too-fast}, no real can be $1$-energy random.

For the case $s=0$, it is conventional to define the Riesz $s$-potential and Riesz $s$-energy differently than in the $s>0$ case.  This is done via logarithmic potential and logarithmic energy.

\begin{defn}
Let $\mu$ be a measure on $2^\omega$. The \emph{Riesz $\log^k$-potential} of $\mu$ is the function $\RPklog\mu(X) = \int (-\log_2 \rho(X,Y))^k \;d\mu(Y)$. The \emph{Riesz   $\log^k$-energy} of $\mu$ is
\[
\REklog(\mu) = \int \RPklog\mu(X) \;d\mu(X) = \iint (-\log_2 \rho(X,Y))^k \;d\mu(Y)d\mu(X).
\]
In the case $k=1$, we drop the $k$, writing, say, $\log$-energy instead of $\log^k$-energy.
\end{defn}

Define $X\in 2^\omega$ to be \emph{$\log^k$-energy random} if it is random for some $\mu$ such that $\REklog(\mu)<\infty$. Again, we can fit $\log^k$-energy randomness into the framework of $f$-energy randomness.

\begin{lem}
$\RPklog\mu(X) = \Pf\mu(X)$ for $f(n) = n^k-{(n-1)^k}$ (where we take $f(0)=0$).
\end{lem}
\begin{proof}
We can rewrite $(-\log_2 \rho(X,Y))^k$ as
\[
\sum_{n=1}^\infty \left(n^k-{(n-1)^k}\right)\chi_{[X\uh n]}(Y).
\]
Therefore,
\begin{align*}
\RPklog\mu(X) &= \int (-\log_2 \rho(X,Y))^k \;d\mu(Y) \\
	&= \int \sum_{n=1}^\infty \left(n^k-{(n-1)^k}\right)\chi_{[X\uh n]}(Y) \;d\mu(Y) \\
	&= \sum_{n=1}^\infty \left(n^k-{(n-1)^k}\right) \int\chi_{[X\uh n]}(Y) \;d\mu(Y) \\
	&= \sum_{n=1}^\infty \left(n^k-{(n-1)^k}\right) \mu[X\uh n].\qedhere
\end{align*}
\end{proof}

Since $n^k-{(n-1)^k}\sim n^{k-1}$, we have the following:

\begin{prop}
$\REklog(\mu)<\infty$ if and only if $\Ef(\mu)<\infty$ for $f(n)=n^{k-1}$.
\end{prop}

\begin{cor}\label{cor:log-energy-equiv}
$X\in 2^\omega$ is $\log^k$-energy random if and only if it is $f$-energy random for $f(n)=n^{k-1}$.
\end{cor}

For Riesz $s$-potential, the $s=0$ case is conventionally defined to be the \emph{$\log$-potential}, that is Riesz $\log^k$-potential for $k=1$.  By the previous corollary, $X$ is $\log$-energy random if and only if $X$ is $f$-energy random for $f(n)=1$.  This agrees with the characterization of $s$-energy randomness as $f$-energy randomness for $f(n)=2^{sn}$.

However, note that $\log$-energy randomness is not equivalent to being random for some $\mu$ with $\iint \rho(X,Y)^{-0} \;d\mu(X) d\mu(Y) < \infty$.  Indeed, every $X$ has this property, since every $X$ is random for some $\mu$ and
\[
\iint \rho(X,Y)^{-0} \;d\mu(X) d\mu(Y) = \iint 1 \;d\mu(X) d\mu(Y) = \mu^2[\lambda] < \infty.
\]
On the other hand, if $X$ is computable, then any $\mu$ for which $X$ is $\mu$-random must have $X$ as an atom.  If $\mu$ has atoms, then $\RElog(\mu)=\infty$. Therefore, computable $X$ cannot be $\log$-energy random.

%%%%%%%%
%%%%%%%%
\section{Capacity}
\label{sec:capacity}
%%%%%%%%
%%%%%%%%

Once again, fix a function $f\colon\omega\to[0,\infty)$. In this section, we investigate $f$-capacity, which is a way of assigning weights to subsets of $2^\omega$. Following Remark~\ref{rem:too-fast}, assume that $|f| = \sum_{n\in\omega} f(n)2^{-n} < \infty$. (It is not hard to see that if $|f| = \infty$, then all sets have $f$-capacity $0$, so this is not an interesting case.) We also assume that $f$ has infinite support (and so $|f|>0$). This is safe by Remark~\ref{rem:fin-support}.

Let $\+M$ be the space of finite Borel measures on $2^\omega$.

\begin{defn}
The \emph{$f$-capacity} of $U\subseteq 2^\omega$ is
\[
\Cf(U) = \inf\;\{\mu(2^\omega)\colon \mu\in\+M\text{ and }(\forall X\in U)\; \Pf\mu(X)\geq 1\}.
\]
Say that $\mu$ \emph{realizes} $\Cf(U)$ if $\mu(2^\omega) = \Cf(U)$ and $(\forall X\in U)\; \Pf\mu(X)\geq 1$.
\end{defn}

We start by proving some basic properties of $\Cf$.

\begin{lem}\label{lem:cap-basic}
Let $U,V\subseteq 2^\omega$.
\begin{enumerate}
	\item\label{it:empty} $\Cf(\emptyset) = 0$.
	\item\label{it:monotone} If $U\subseteq V$, then $\Cf(U)\leq\Cf(V)$.
	\item\label{it:geq-sup} $\displaystyle\Cf(U) \geq \sup\;\{\nu(U)\colon \nu\in\+M\text{ and }(\forall Y)\; \Pf\nu(Y)\leq 1\}$.
	\item\label{it:full} $\displaystyle\Cf(2^\omega) = 1/|f|$. Moreover, $\Cf(2^\omega)$ is realized by the uniform measure $\mu_{1/|f|}$ such that $\mu_{1/|f|}(2^\omega) = 1/|f|$.
\end{enumerate}
\end{lem}
\begin{proof}
\eqref{it:empty} and \eqref{it:monotone} are immediate from the definition of $\Cf$.

For \eqref{it:geq-sup}, let $\mu$ be any measure such that $(\forall X\in U)\; \Pf\mu(X)\geq 1$ and let $\nu$ be any measure such that $(\forall Y)\; \Pf\nu(Y)\leq 1$. Using Lemma~\ref{lem:two-measures},
\[
\mu(2^\omega) \geq \int \Pf\nu(Y) \;d\mu(Y) = \sum_{\sigma\in 2^{<\omega}} f(|\sigma|)\mu[\sigma]\nu[\sigma] = \int \Pf\mu(X) \;d\nu(X) \geq \nu(U).
\]
So $\Cf(U)\geq \sup\;\{\nu(U)\colon (\forall Y)\; \Pf\nu(Y)\leq 1\}$.

For \eqref{it:full}, let $\mu_{1/|f|}$ be the uniform measure such that $\mu_{1/|f|}(2^\omega)=1/|f|$. In other words, $\mu_{1/|f|}[\sigma] = 2^{-|\sigma|}/|f|$. Note that for all $X\in 2^\omega$ we have
\[
\Pf\mu_{1/|f|}(X) = \sum_{n\in\omega} f(n)\mu_{1/|f|}[X\uh n] = \frac{1}{|f|}\sum_{n\in\omega} f(n)2^{-n} = 1.
\]
Therefore, by definition $\Cf(2^\omega) \leq \mu_{1/|f|}(2^\omega) = 1/|f|$. But by \eqref{it:geq-sup} we also have $\Cf(2^\omega) \geq \mu_{1/|f|}(2^\omega) = 1/|f|$. 
\end{proof}

In order to give inductive proofs and and recursive definitions involving $f$-capacity, we need notation for the capacities associated to the shifts of $f$. For $k\in\omega$, let $f_k$ be the function defined by $f_k(n) = f(n+k)$. Note that $|f_k|\leq 2^k|f|<\infty$ and $f_k$ has infinite support, so our assumptions about $f$ are preserved.

If $U\subseteq 2^\omega$, let $U_i = \{X\colon iX\in U\}$. The next lemma give us a relationship between $\Cf[k](U)$, $\Cf[k+1](U_0)$ and $\Cf[k+1](U_1)$.

\begin{lem}\label{lem:cap-relation}
For $k\in\omega$, $\displaystyle\Cf[k](U) = \frac{\Cf[k+1](U_0)+\Cf[k+1](U_1)}{1+f(k)\left(\Cf[k+1](U_0)+\Cf[k+1](U_1)\right)}$.
\end{lem}
\begin{proof}
It is sufficient to prove the lemma for $k=0$. Note that $\Cf(U)\leq \Cf(2^\omega) = 1/|f| < 1/f(0)$. So consider a measure $\mu$ such that $(\forall X\in U)\; \Pf\mu(X)\geq 1$ and $\mu[\lambda] < 1/f(0)$. Define $\nu_0$ such that $\nu_0[\sigma] = \mu[0\sigma]$. If $X\in U_0$, then
\begin{align*}
1 \leq \Pf\mu(0X) &= \sum_{n\in\omega} f(n)\mu[(0X)\uh n] = f(0)\mu[\lambda] + \sum_{n\in\omega} f(n+1)\mu[(0X)\uh (n+1)] \\
	&= f(0)\mu[\lambda] + \sum_{n\in\omega} f(n+1)\nu_0[X\uh n] = f(0)\mu[\lambda] + \Pf[1]\nu_0(X).
\end{align*}
So $\Pf[1]\nu_0(X)\geq 1-f(0)\mu[\lambda]$ for all $X\in U_0$. By assumption, $1-f(0)\mu[\lambda]>0$. Therefore, the measure $\mu_0 = \left(\frac{1}{1-f(0)\mu[\lambda]}\right)\nu_0$ has the property that $\Pf[1]\mu_0(X)\geq 1$ for all $X\in U_0$. Hence
\begin{align*}
\mu_0[\lambda] &= \left(\frac{1}{1-f(0)\mu[\lambda]}\right)\nu_0[\lambda] \geq \Cf[1](U_0), \\
\text{and so } \mu[0] = \nu_0[\lambda] &\geq (1-f(0)\mu[\lambda])\Cf[1](U_0).
\end{align*}
Similarly, $\mu[1] \geq (1-f(0)\mu[\lambda])\Cf[1](U_1)$. Adding these together, we get
\[
\mu[\lambda] \geq (1-f(0)\mu[\lambda])(\Cf[1](U_0)+\Cf[1](U_1)).
\]
Solving for $\mu[\lambda]$,
\[
\mu[\lambda] \geq \frac{\Cf[1](U_0)+\Cf[1](U_1)}{1+f(0)\left(\Cf[1](U_0)+\Cf[1](U_1)\right)}.
\]
This is true for any $\mu$ such that $(\forall X\in U)\; \Pf\mu(X)\geq 1$ and $\mu[\lambda]<1/f(0)$. But recall that $\Cf(U)<1/f(0)$, so
\[
\Cf(U)\geq \frac{\Cf[1](U_0)+\Cf[1](U_1)}{1+f(0)\left(\Cf[1](U_0)+\Cf[1](U_1)\right)}.
\]

For the other direction, we will build a measure $\mu$ approximating $\Cf(U)$. For each $i\in\{0,1\}$, let $\mu_i$ be a measure such that $\Pf[1]\mu_i(X)\geq 1$ for all $X\in U_i$. We define $\mu$ such that
\[
\mu[i\sigma] = \frac{\mu_i[\sigma]}{1+f(0)\left(\mu_0[\lambda]+\mu_1[\lambda]\right)}.
\]
In particular, $\displaystyle\mu[\lambda] = \frac{\mu_0[\lambda]+\mu_1[\lambda]}{1+f(0)\left(\mu_0[\lambda]+\mu_1[\lambda]\right)}$. Now take $X\in U_0$. We have
\begin{align*}
\Pf\mu(0X) &= f(0)\mu[\lambda] + \frac{\Pf[1]\mu_0(X)}{1+f(0)\left(\mu_0[\lambda]+\mu_1[\lambda]\right)} \\
	&\geq f(0)\left(\frac{\mu_0[\lambda]+\mu_1[\lambda]}{1+f(0)\left(\mu_0[\lambda]+\mu_1[\lambda]\right)}\right) + \frac{1}{1+f(0)\left(\mu_0[\lambda]+\mu_1[\lambda]\right)} = 1.
\end{align*}
Similarly, if $X\in U_1$, then $\Pf\mu(1X)\geq 1$. Hence for any $X\in U$ we have $\Pf\mu(X)\geq 1$. Therefore, 
\[
\Cf(U) \leq \mu[\lambda] = \frac{\mu_0[\lambda]+\mu_1[\lambda]}{1+f(0)\left(\mu_0[\lambda]+\mu_1[\lambda]\right)}.
\]
Taking $\mu_0[\lambda]\to\Cf[1](U_0)$ and $\mu_1[\lambda]\to\Cf[1](U_1)$ proves that
\[
\Cf(U)\leq \frac{\Cf[1](U_0)+\Cf[1](U_1)}{1+f(0)\left(\Cf[1](U_0)+\Cf[1](U_1)\right)}. \qedhere
\]
\end{proof}

\begin{rem}\label{rem:Cs}
Fix $s\in[0,1)$ and let $f(n)=2^{sn}$. Define the \emph{$s$-capacity} of a set $U\subseteq 2^\omega$ to be $\Cs(U)=\Cf(U)$. Note that $\Pf[1]\mu = 2^s\Pf\mu$ for any measure $\mu$. Hence, we have $\Cf[1] = \Cf/2^s = \Cs/2^s$. Combining this observation with Lemma~\ref{lem:cap-relation} gives us a nice recursive expression for $\Cs$. For any $U\subseteq 2^\omega$,
\begin{align*}
\Cs(U) &= \frac{\Cf[1](U_0)+\Cf[1](U_1)}{1+f(0)\left(\Cf[1](U_0)+\Cf[1](U_1)\right)} = \frac{\Cs(U_0)/2^s+\Cs(U_1)/2^s}{1+1\cdot\left(\Cs(U_0)/2^s+\Cs(U_1)/2^s\right)}\\
	&= \frac{\Cs(U_0)+\Cs(U_1)}{2^s+\Cs(U_0)+\Cs(U_1)}.
\end{align*}
\end{rem}

Note that the $f$-capacity of every clopen set is determined by Lemma~\ref{lem:cap-relation} together with parts~\eqref{it:empty} and~\eqref{it:full} of Lemma~\ref{lem:cap-basic}.\footnote{In fact, if $|f|$ is computable (which implies that $f$ is also computable), then we could compute the $f$-capacity of a clopen set. We defer effectiveness questions until later.} Moreover, if $U$ is a clopen set and $k\in\omega$, then there is a measure $\mu$ realizing $\Cf[k](U)$. To see this, note that $\Cf[k](\emptyset)$ is realized by the empty measure and $\Cf[k](2^\omega)$ is realized by the measure $\mu_{1/|f_k|}$ from Lemma~\ref{lem:cap-basic}. Furthermore, if $\Cf[k+1](U_0)$ and $\Cf[k+1](U_1)$ are realized by $\mu_0$ and $\mu_1$, respectively, then $\Cf[k](U)$ is realized by the measure $\mu$ defined by
\[
\mu[i\sigma] = \frac{\mu_i[\sigma]}{1+f(k)\left(\mu_0[\lambda]+\mu_1[\lambda]\right)} = \frac{\mu_i[\sigma]}{1+f(k)\left(\Cf[k+1](U_0)+\Cf[k+1](U_1)\right)}.
\]
This follows by the same argument as the second part of the proof of Lemma~\ref{lem:cap-relation}. So we can produce a measure realizing $\Cf[k](U)$ for any clopen set $U$. A straightforward calculation shows that if $\mu_0 = \mu_1 = \mu_{1/|f_{k+1}|}$, then $\mu$ is $\mu_{1/|f_k|}$. This means that we get the same measure $\mu$ no matter how we decompose the clopen set $U$. (In fact, it turns out that at most one measure can realize $\Cf[k](U)$, but we do not need this.)

\begin{lem}\label{lem:cap-clopen}
Let $U\subseteq 2^\omega$ be a clopen set and $k\in\omega$. Let $\mu$ be the measure realizing $\Cf[k](U)$ described above.
\begin{enumerate}
	\item\label{it:total-measure} $\mu(U) = \Cf[k](U)$.
	\item\label{it:equal-one} $(\forall X\in U)\; \Pf[k]\mu(X) = 1$.
	\item\label{it:leq-one} $(\forall X)\; \Pf[k]\mu(X)\leq 1$.
\end{enumerate}
\end{lem}
\begin{proof}
\eqref{it:total-measure} follows by induction on the definition of $\mu$. It is clearly true for $U=\emptyset$ and $U=2^\omega$. For any other case, assume that $\mu_0(U_0) = \Cf[k+1](U_0)$ and $\mu_1(U_1) = \Cf[k+1](U_1)$. Then
\begin{align*}
\mu(U) &= \mu(0U_0) + \mu(1U_1) \\
	&= \frac{\mu_0(U_0)}{1+f(k)\left(\Cf[k+1](U_0)+\Cf[k+1](U_1)\right)} + \frac{\mu_1(U_1)}{1+f(k)\left(\Cf[k+1](U_0)+\Cf[k+1](U_1)\right)} \\
	&= \frac{\Cf[k+1](U_0)+\Cf[k+1](U_1)}{1+f(k)\left(\Cf[k+1](U_0)+\Cf[k+1](U_1)\right)} = \Cf[k](U).
\end{align*}

We also prove $\eqref{it:equal-one}$ by induction on the definition of $\mu$. Again, the claim is clear for $U=\emptyset$ and $U=2^\omega$. For any other case, consider $X\in U$. Without loss of generality, assume that the first bit of $X$ is $0$. Say that $X = 0Y$, so $Y\in U_0$. Let $\mu_0$ be the measure that we construct to realize $\Cf[k+1](U_0)$. By induction, $\Pf[k+1]\mu_0(Y) = 1$. Therefore,
\begin{align*}
\Pf[k]\mu(X) &= f(k)\mu[\lambda] + \Pf[k+1]\left(\frac{\mu_0}{1+f(k)\left(\Cf[k+1](U_0)+\Cf[k+1](U_1)\right)}\right)(Y) \\
	&= f(k)\Cf[k](U) + \left(\frac{1}{1+f(k)\left(\Cf[k+1](U_0)+\Cf[k+1](U_1)\right)}\right)\Pf[k+1]\mu_0(Y) \\
	&= f(k)\left(\frac{\Cf[k+1](U_0)+\Cf[k+1](U_1)}{1+f(k)\left(\Cf[k+1](U_0)+\Cf[k+1](U_1)\right)}\right) \\
	&\hspace{1in} + \frac{1}{1+f(k)\left(\Cf[k+1](U_0)+\Cf[k+1](U_1)\right)} = 1.
\end{align*}

To prove \eqref{it:leq-one}, consider an $X\in 2^\omega$. If $X\in U$, then we have already shown that $\Pf[k]\mu(X) = 1$. If $U = \emptyset$, then $\mu$ is the zero measure and so $\Pf[k]\mu(X) = 0 \leq 1$. Otherwise, let $m>0$ be least such that $[X\uh m]\cap U = \emptyset$. Note that $\mu[X\uh m] = 0$.  (This follows from \eqref{it:total-measure}, $\mu(U)=\Cf(U)=\mu(2^\omega)$.) Pick $Z$ such that $Z\uh m-1 = X\uh m-1$ and $Z\in U$. Then using part~\eqref{it:equal-one},
\begin{align*}
\Pf[k]\mu(X) &= \sum_{n\in\omega} f_k(n)\mu[X\uh n] = \sum_{n=0}^{m-1} f_k(n)\mu[X\uh n] \\
	&= \sum_{n=0}^{m-1} f_k(n)\mu[Z\uh n] \leq \Pf[k]\mu(Z) = 1. \qedhere
\end{align*}
\end{proof}

We now give two useful expressions for $\Cf(U)$ in the case that $U$ is open.

\begin{lem}\label{lem:cap-open}
Let $U\subseteq 2^\omega$ be an open set.
\begin{enumerate}
	\item\label{it:approx} $\Cf(U) = \sup\;\{\Cf(V)\colon V\subseteq U\text{ is clopen}\}$.
	\item\label{it:eq-sup} $\Cf(U) = \sup\;\{\nu(U)\colon \nu\in\+M\text{ and }(\forall Y)\; \Pf\nu(Y)\leq 1\}$.
\end{enumerate}
\end{lem}
\begin{proof}
To prove~\eqref{it:approx}, first note that Lemma~\ref{lem:cap-basic}\eqref{it:monotone} implies that
\[
\Cf(U) \geq \sup\;\{\Cf(V)\colon V\subseteq U\text{ is clopen}\}.
\]
Now let $U_0\subseteq U_1\subseteq \cdots$ be a sequence of clopen sets such that $U = \bigcup_{i\in\omega} U_i$. Let $\mu_0, \mu_1, \dots$ be the corresponding measures, as defined above. There is a subsequence $\mu_{i_0}, \mu_{i_1}, \dots$ that converges on all clopen sets. (In fact, it can be shown that the full sequence converges, but this is not needed.) Let $\mu$ be the limit.

Consider $X\in U$. Fix $\sigma\prec X$ such that $[\sigma]\subseteq U$. Take $K\in\omega$ large enough that $[\sigma]\subseteq U_{i_K}$. For any $k\geq K$ and any $\tau\succeq\sigma$, we have $\mu_{i_k}[\tau] = 2^{|\sigma|-|\tau|}\mu_{i_k}[\sigma]$. This is preserved in the limit, so it is true of $\mu$. Therefore,
\begin{align*}	
\Pf\mu(X) &= \sum_{n\in\omega} f(n)\mu[X\uh n] = \sum_{n<|\sigma|} f(n)\mu[X\uh n] + \sum_{n\geq|\sigma|} f(n)\mu[X\uh n] \\
	&= \sum_{n<|\sigma|} f(n)\mu[X\uh n] + \sum_{n\geq|\sigma|} f(n)2^{|\sigma|-n}\mu[\sigma] \\
	&= \lim_{k\to\infty} \left(\sum_{n<|\sigma|} f(n)\mu_{i_k}[X\uh n] + \sum_{n\geq|\sigma|} f(n)2^{|\sigma|-n}\mu_{i_k}[\sigma]\right) \\
	&= \lim_{k\to\infty} \Pf\mu_{i_k}(X) = \lim_{k\to\infty} 1 = 1.
\end{align*}
(The key to this calculation is that the limit commutes with $\Pf$ in the expression $\lim_{k\to\infty}\Pf\mu_{i_k}(X)$ because, in this case, $\Pf\mu_{i_k}(X)$ only depends on a fixed finite number of values of $\mu_{i_k}$.) We have shown that $(\forall X\in U)\; \Pf\mu(X)\geq 1$. This implies that $\Cf(U) \leq \mu[\lambda] = \lim_{k\to\infty} \mu_{i_k}[\lambda] = \lim_{k\to\infty} \Cf(U_{i_k})$. Therefore, 
\[
\Cf(U) \leq \sup\;\{\Cf(V)\colon V\subseteq U\text{ is clopen}\}.
\]

For~\eqref{it:eq-sup}, let $r = \sup\;\{\nu(U)\colon \nu\in\+M\text{ and }(\forall Y)\; \Pf\nu(Y)\leq 1\}$. In Lemma~\ref{lem:cap-basic}\eqref{it:geq-sup}, we proved that $\Cf(U) \geq r$. Now take any clopen $V\subseteq U$. Let $\nu$ be the measure that we constructed to realize $\Cf(V)$. By Lemma~\ref{lem:cap-clopen}, $(\forall Y)\; \Pf\nu(Y)\leq 1$ and, of course, $\nu(V) = \Cf(V)$. Therefore, $\Cf(V)\leq r$. By~\eqref{it:approx}, we now have $\Cf(U) = \sup\;\{\Cf(V)\colon V\subseteq U\text{ is clopen}\} \leq r$.
\end{proof}

%%%%%%%%
\subsection*{\texorpdfstring{\boldmath$\Cf$}{C\textunderscore f}-randomness}
%%%%%%%%

We can view the $f$-capacity of a set as a measure of its size. Using this size in place of measure gives us a notion of randomness analogous to Martin-L\"of randomness.

\begin{defn}
A computable sequence $\{U_n\}_{n\in\omega}$ of $\Sigma^0_1$-classes is a \emph{$\Cf$-test} if $(\forall n)\; \Cf(U_n)\leq 2^{-n}$. A sequence $X\in 2^\omega$ \emph{passes} the $\Cf$-test if $X\notin\bigcup_{n\in\omega} U_n$. We say that $X\in 2^\omega$ is \emph{$\Cf$-random} if it passes every $\Cf$-test.
\end{defn}

The second author characterized $f$-energy randomness as $\Cf$-randomness, at least under the assumption that $\Cf$ is \emph{computable}, i.e., uniformly computable on clopen sets \cite{Rute:aa}. We will prove it under the slightly weaker assumption that $f$ is computable. Note that even if $f$ is computable, that does not mean that $\Cf$ is computable: $\Cf(2^\omega) = 1/|f|$ is guaranteed only to be a right-c.e.\ real number. It should be noted, however, that the weaker assumption adds no difficulty to the proof, which we include for the sake of completeness.

\begin{thm}[Rute~\cite{Rute:aa}]\label{thm:rute}
Assume that $f\colon\omega\to[0,\infty)$ is computable. Then $X\in 2^\omega$ is $f$-energy random if and only if it is $\Cf$-random.
\end{thm}
\begin{proof}
First assume that $X$ is not $\Cf$-random. Then $X$ fails a $\Cf$-test $\{U_n\}_{n\in\omega}$. Let $\mu$ be any measure such that $\Ef\mu<\infty$. We must prove that $X$ is not random for $\mu$. Fix $c$ such that $\Ef\mu = \int\Pf\mu(X)\; d\mu(X) < c$. Fix $n$ and let $V_n = \{X\colon \Pf\mu(X)>2^n\}$. So $\mu(V_n) < c2^{-n}$. Define a new measure $\mu_n$ by $\mu_n(U)=\mu(U\smallsetminus V_n)$. Note that $\mu_n\leq\mu$, so $\Pf\mu_n\leq\Pf\mu$. Hence if $\Pf\mu(X)\leq 2^n$, then $\Pf\mu_n(X)\leq 2^n$. If $\Pf\mu(X)> 2^n$, then there is a least $\sigma\prec X$ such that $\sum_{\tau\preceq\sigma} f(|\tau|)\mu[\tau]>2^n$. Note that $[\sigma]\subseteq V_n$, so $\mu_n[\sigma]=0$. Therefore,
\[
\Pf\mu_n(X) = \sum_{m\in\omega} f(n)\mu_n[X\uh m] = \sum_{\tau\prec\sigma} f(|\tau|)\mu_n[\tau] \leq \sum_{\tau\prec\sigma} f(|\tau|)\mu[\tau] \leq 2^n.
\]
So $(\forall X)\; \Pf\mu_n(X)\leq 2^n$. By Lemma~\ref{lem:cap-open}\eqref{it:eq-sup}, for any open set $U$ and any $n\in\omega$:
\[
\mu(U) = \mu(U\smallsetminus V_n) + \mu(U\cap V_n) < \mu_n(U) + c2^{-n} \leq 2^n\Cf(U) + c2^{-n}.
\]
Thus $\mu(U_{2n}) < 2^n\Cf(U_{2n}) + c2^{-n} \leq (1+c)2^{-n}$, meaning that $\{U_{2n}\}_{n\in\omega}$ is (essentially) a $\mu$-test covering $X$. This proves that $X$ is not random for $\mu$, but $\mu$ was any measure such that $\Ef\mu<\infty$, so $X$ is not $f$-energy random.

For the other direction, assume that $X$ is not $f$-energy random. We can produce a universal $\nu$-test uniformly in a measure $\nu$. Let $\{U^\nu_n\}_{n\in\omega}$ be the resulting tests. Note that the computability of $f$ implies that $\{\nu\in\+M\colon (\forall Y)\; \Pf\nu(Y)\leq 1\}\}$ is a $\Pi^0_1$-class. For each $n$, let $V_n = \bigcap\; \{U^\nu_n\colon \nu\in\+M\text{ and }(\forall Y)\; \Pf\nu(Y)\leq 1\}$. Because $V_n$ is the intersection of a uniform family of $\Sigma^0_1[\nu]$-classes over a (compact) $\Pi^0_1$-class of measures, it is a $\Sigma^0_1$-class. This is uniform in $n$. By Lemma~\ref{lem:cap-open}\eqref{it:eq-sup}:
\begin{align*}
\Cf(V_n) &= \sup\;\{\nu(V_n)\colon \nu\in\+M\text{ and }(\forall Y)\; \Pf\nu(Y)\leq 1\} \\
	&\leq \sup\;\{\nu(U^\nu_n)\colon \nu\in\+M\text{ and }(\forall Y)\; \Pf\nu(Y)\leq 1\} \leq 2^{-n}.
\end{align*}
Therefore, $\{V_n\}_{n\in\omega}$ is a $\Cf$-test. Finally, if $(\forall Y)\; \Pf\nu(Y)\leq 1$, then $\Ef\nu\leq 1$. In this case, $X$ is not random for $\nu$, so $X\in U^\nu_n$ for all $n$. This implies that $X\in V_n$ for all $n$, so $X$ is not $\Cf$-random.
\end{proof}

%%%%%%%%
%%%%%%%%
\section{How effective is \texorpdfstring{$f$}{f}-capacity?}
\label{sec:effectiveness}
%%%%%%%%
%%%%%%%%

Let $S\subseteq 2^{<\omega}$ be a prefix-free c.e.\ set. In particular, let $S_0\subseteq S_1\subseteq S_2\cdots$ be a computable sequence of finite sets such that $S = \bigcup_{t\in\omega} S_t$. In the previous section, we noted that even if $f$ is computable, $\Cf$ need not be computable (on clopen sets). So assume that $|f|$ is computable in addition to $f$ being computable. This also implies that the sequence $\{|f_k|\}_{k\in\omega}$ is computable. From this, we can compute $\Cf[][S_t]$ by using Lemma~\ref{lem:cap-basic}\eqref{it:empty} and~\eqref{it:full} for the base cases and Lemma~\ref{lem:cap-relation} for the recursion. By Lemma~\ref{lem:cap-open}\eqref{it:approx} (and the monotonicity of $s$-capacity), $\Cf[][S] = \sup_{t\in\omega} \Cf[][S_t]$, hence $\Cf[][S]$ is a left-c.e.\ real. This is as effective as we could hope $\Cf[][S]$ to be for an arbitrary $\Sigma^0_1$-class $[S]$.

However, even assuming that $|f|$ is computable, there is a sense in which $f$-capacity is not as effective as we would like. In the proof of Lemma~\ref{lem:cap-open}, we constructed a measure $\mu$ that realizes $\Cf[][S]$. Ideally, for a c.e.\ set $S$ we would want $\mu$ to be left-c.e., but this need not be the case. An example will clarify the problem. Fix $s\in[0,1)$. Using Lemma~\ref{lem:cap-basic}\eqref{it:full}, it is not hard to see that the unique measure that realizes $\Cs(2^\omega)=1-2^{s-1}$ is the uniform measure $\mu$ with $\mu[\lambda] = 1-2^{s-1}$. Similarly, by Remark~\ref{rem:Cs}, the unique measure $\mu_0$ that realizes $\Cs[0]$ is the one that is uniform on $[0]$ and has
\[
\mu_0[\lambda] = \mu_0[0] = \frac{1-2^{s-1}}{1+2^{s-1}}.
\]
Note that $\mu[0] = (1-2^{s-1})/2 < (1-2^{s-1})/(1+2^{s-1}) = \mu_0[0]$. This simple example illustrates that even if $S_u\subseteq S_v$, a measure that realizes $\Cf[][S_u]$ might unavoidably give some sets higher measure than a measure that realizes $\Cf[][S_v]$. In other words, as we approximate $[S]$ we get a natural sequence of approximations to the measure that realizes $\Cf[][S]$, but there is no way to guarantee that these approximations converge from below.

This problem is surmountable: we will show that if $f$ is sufficiently well-behaved and $S\subseteq 2^{<\omega}$ is a c.e.\ set, then there is a left-c.e.\ measure that ``almost realizes'' $\Cf[][S]$.

%%%%%%%%
\subsection*{Our assumptions about \texorpdfstring{\boldmath$f$}{f}}
%%%%%%%%

As in the previous section, we assume that $|f|<\infty$ and $f$ has infinite support. These assumptions, which are important if we want to apply the work of Section~\ref{sec:capacity}, are used implicitly below; by Remarks~\ref{rem:too-fast} and~\ref{rem:fin-support}, they do not limit the scope of our main result. We also assume that $|f|$ is computable. We will need one more assumption.

\begin{defn}
We say that $f\colon\omega\to[0,\infty)$ is \emph{amicable} if $\displaystyle\sup_{k\in\omega} f(k)\Cf[k+1](2^\omega)<\infty$.
\end{defn}

Note that $\Cf[k+1](2^\omega) = 1/|f_{k+1}| < 1/f(k+1)$, so it is immediate that
\begin{itemize}
	\item If $\sup_{k\in\omega} f(k)/f(k+1)<\infty$, then $f$ is amicable.
	\item If $f$ is nondecreasing, then $f$ is amicable.
\end{itemize}

Assuming that $f$ is amicable, let $\|f\| = \displaystyle\sup_{k\in\omega} f(k)\Cf[k+1](2^\omega)$.

%%%%%%%%
\subsection*{Dynamic weight}
%%%%%%%%

As above, let $S\subseteq 2^{<\omega}$ be a prefix-free c.e.\ set effectively approximated by an increasing sequence of finite sets $\{S_t\}_{t\in\omega}$. We may assume that $S_0=\emptyset$ and $(\forall t)\; |S_{t+1}\smallsetminus S_t|\leq 1$. Call such an enumeration \emph{good}. Let $a = 2\|f\|+2$. Our goal is to produce a left-c.e.\ measure $\mu_S$ such that $\mu_S[\lambda]\leq a\Cf[][S]$ and $(\forall X\in [S])\; \Pf\mu_S(X)\geq 1$.

First, we define a ``dynamic weight'' for prefix-free c.e.\ sets (or rather, for good enumerations of such sets). For each $k\in\omega$, let $\ww^k_0(S) = 0$. Assume that we have already defined $\ww^k_{t-1}(S)$. If $S_t=S_{t-1}$, then let $\ww^k_t(S) = \ww^k_{t-1}(S)$. Otherwise, let $\sigma$ be the unique string in $S_t\smallsetminus S_{t-1}$. If $\sigma=\lambda$, then let $\ww^k_t(S) = 1/|f_k|$. For any other $\sigma$, assume by induction on the length of $\sigma$ that $\ww^{k+1}_t(S^0)$ and $\ww^{k+1}_t(S^1)$ are defined, where $S^i = \{\sigma\colon i\sigma\in S\}$ and $S^i$ is given the induced good enumeration $S^i_t = \{\sigma\colon i\sigma\in S_t\}$. Let
\[
\ww^k_t(S) = \ww^k_{t-1}(S) + \frac{(\ww^{k+1}_t(S^0)+\ww^{k+1}_t(S^1))-(\ww^{k+1}_{t-1}(S^0)+\ww^{k+1}_{t-1}(S^1))}{1+f(k)\left(\ww^{k+1}_t(S^0)+\ww^{k+1}_t(S^1)\right)}.
\]
A simple induction shows that $\ww^k_t(S)$ is nondecreasing as a function of $t$. Therefore, we can define $\ww^k(S) = \lim_{t\to\infty} \ww^k_t(S)$.  

It should be noted that the value of $\ww^k(S)$ depends not only on the prefix-free set $S$, but also on the choice of good enumeration $\{S_t\}_{t\in\omega}$.  Nonetheless, this next lemma shows that $\ww^k(S)$ is not much larger than $\Cf[k][S]$.

\begin{lem}\label{lem:bound}
Fix a (good enumeration of a) prefix-free c.e.\ set $S\subseteq 2^{<\omega}$. For each $k\in\omega$, we have $\ww^k(S)\leq (2\|f\|+2)\Cf[k][S]$.
\end{lem}

We will need a simple inequality.

\begin{lem}\label{lem:calc}
For any $a\geq 0$, if $y\in[0,a-2]$, then $\displaystyle\ln(1+ay) \leq a\left(\frac{y}{1+y}\right)$.
\end{lem}
\begin{proof}
Let $f(y) = ay/(1+y) - \ln(1+ay)$. We must show that $f(y)\geq 0$ on $[0,a-2]$. If $y\in[0,a-2]$, then
\[
f'(y) = \frac{a}{(1+y)^2}-\frac{a}{1+ay} = \frac{ay(a-2-y)}{(1+y)^2(1+ay)} \geq 0.
\]
Hence $f(y)$ is increasing on $[0,a-2]$. Since $f(0) = 0$, the lemma holds.
\end{proof}

\begin{proof}[Proof of Lemma~\ref{lem:bound}]
In light of Lemma~\ref{lem:cap-open}\eqref{it:approx} (or even Lemma~\ref{lem:cap-basic}\eqref{it:monotone}), it is enough to prove the assertion assuming that $S$ is finite. We do this by induction on the length of the longest string in $S$.

Let $a=2\|f\|+2$. If $S=\emptyset$, then $\ww^k(S)=0=a\Cf[k][S]$. If $S=\{\lambda\}$, then $\ww^k(S) = 1/|f_k| < a/|f_k| = a\Cf[k][S]$. In any other case, assume by induction that $\ww^{k+1}(S^0)\leq a\Cf[k+1][S^0]$ and $\ww^{k+1}(S^1)\leq a\Cf[k+1][S^1]$ (where the dynamic weights of $S^0$ and $S^1$ are defined using the enumerations induced by the enumeration of $S$). Note that if $f(k)=0$, then $\ww^k(S)=\ww^{k+1}(S^0)+\ww^{k+1}(S^1)\leq a\Cf[k+1][S^0]+a\Cf[k+1][S^1]=a\Cf[k][S]$, where the last equality follows from Lemma~\ref{lem:cap-relation}, so we are done. Therefore, we may assume that $f(k)>0$.

Let $t_0 = 0$, and let $t_1 < t_2 < \cdots < t_m$ be the sequence of stages $t$ at which $\ww^{k+1}_t(S^0) + \ww^{k+1}_t(S^1)$ increases. For each $i\in\{0,\dots,m\}$, let $v_i = \ww^{k+1}_{t_i}(S^0) + \ww^{k+1}_{t_i}(S^1)$. In particular, $v_0=0$ and $v_m = \ww^{k+1}(S^0) + \ww^{k+1}(S^1)$. Note that if $0<i\leq m$, then at stage $t_i$ we add $(v_i-v_{i-1})/(1+f(k)v_i)$ to $\ww^k(S)$. In other words, 
\[
\ww^k(S) = \sum_{i=1}^m \frac{v_i-v_{i-1}}{1+f(k)v_i}.
\]
But note that
\[
\frac{v_i-v_{i-1}}{1+f(k)v_i} = \int_{v_{i-1}}^{v_i} \frac{dx}{1+f(k)v_i} \leq \int_{v_{i-1}}^{v_i} \frac{dx}{1+f(k)x}.
\]
Let $z = \Cf[k+1][S^0] + \Cf[k+1][S^1]$. So by induction, $v_m\leq az$. We have
\begin{align*}
\ww^k(S) &= \sum_{i=1}^m \frac{v_i-v_{i-1}}{1+f(k)v_i} \leq \sum_{i=1}^m \int_{v_{i-1}}^{v_i} \frac{dx}{1+f(k)x} = \int_0^{v_m} \frac{dx}{1+f(k)x} \\
	&= \frac{1}{f(k)}\ln(1+f(k)v_m) \leq \frac{1}{f(k)}\ln(1+af(k)z).
\end{align*}
On the other hand, by Lemma~\ref{lem:cap-relation},
\[
\Cf[k][S] = \frac{\Cf[k+1][S^0] + \Cf[k+1][S^1]}{1+f(k)\left(\Cf[k+1][S^0] + \Cf[k+1][S^1]\right)} = \frac{z}{1+f(k)z}.
\]
We have $z\leq 2\Cf[k+1](2^\omega)$, so $f(k)z\leq 2\|f\| = a-2$. Therefore, we can apply Lemma~\ref{lem:calc} with $y=f(k)z$ to obtain
\[
f(k)\ww^k(S) \leq \ln(1+af(k)z) \leq a\left(\frac{f(k)z}{1+f(k)z}\right) = af(k)\Cf[k][S].
\]
Dividing by $f(k)$,
\[
\ww^k(S) \leq (2\|f\|+2)\Cf[k][S]. \qedhere
\]
\end{proof}

Now we are ready to define the measures $\mu^k_S$ associated with (the good enumeration of) $S$. They are defined in stages. For each $k\in\omega$, let $\mu^k_{S,0}$ be the zero measure. If $v = \ww^k_{t+1}(S)-\ww^k_t(S)$ in not zero, then the increase can be attributed to the unique string $\sigma\in S_{t+1}\smallsetminus S_t$. Let $\nu$ be the measure that is uniform on $[\sigma]$ and has $\nu[\lambda] = \nu[\sigma] = v$. Let $\mu^k_{S,t+1} = \mu^k_{S,t} + \nu$. Note that $\{\mu^k_t\}_{t\in\omega}$ is a nondecreasing sequence of computable measures, hence we can define a left-c.e.\ measure $\mu^k_S = \lim_{t\to\infty} \mu^k_{S,t}$. It is clear that $\mu^k_S[\lambda] = \ww^k(S)$. Note also that the construction is uniform in the enumeration of $S$. All that is left is to prove that the $f$-potential of $\mu^k_S$ is at least $1$ for all $X\in [S]$.

\begin{lem}
$(\forall k)(\forall X\in [S])\; \Pf[k]\mu^k_S(X)\geq 1$.
\end{lem}
\begin{proof}
Consider $X\in [S]$ and take $\sigma\in S$ such that $\sigma\prec X$. Let $t$ be the stage at which $\sigma$ enters $S_t$. Note that $\mu^k_{S_t}\leq \mu^k_S$, hence if we prove that $\Pf[k]\mu^k_{S_t}(X)\geq 1$, then $\Pf[k]\mu^k_S(X)\geq 1$. So it is sufficient to prove the following claim: if $S = S_t$ is finite and $\sigma$ enters $S$ at stage $t$, then $(\forall k)(\forall X\in[\sigma])\; \Pf[k]\mu^k_S(X)\geq 1$. Our proof proceeds by induction on the length of $\sigma$.

If $\sigma = \lambda$, then $\mu^k_S$ is just the measure $\mu_{1/|f_k|}$ from Lemma~\ref{lem:cap-basic}\eqref{it:full}, so we are done. If $\sigma\neq\lambda$, then without loss of generality, assume that $\sigma$ starts with $0$. Let $\sigma = 0\tau$ and $X = 0Y$. By our inductive assumption, we have $\Pf[k+1]\mu^{k+1}_{S^0}(Y)\geq 1$.

Define $\nu_0$ such that $\nu_0[\rho] = \mu^k_S[0\rho]$. Consider a string $\rho$ that enters $S^0$ at stage $r\leq t$ (allowing the possibility that $\rho = \sigma$). Let $v = \ww^{k+1}_r(S^0)-\ww^{k+1}_{r-1}(S^0)$. Then at stage $r$ we add measure $v$ to $\mu^{k+1}_{S^0}$ uniformly on $[\rho]$. On the other hand, we add
\[
\frac{v}{1+f(k)\left(\ww^{k+1}_r(S^0)+\ww^{k+1}_r(S^1)\right)} \geq \frac{v}{1+f(k)\left(\ww^{k+1}(S^0)+\ww^{k+1}(S^1)\right)}
\]
to $\mu^k_S$ uniformly on $[0\rho]$. This same measure is added to $\nu_0$ uniformly on $[\rho]$. As this is the only way that measure is added to either $\mu^{k+1}_{S^0}$ or $\nu_0$, we have
\[
\nu_0 \geq \left(\frac{1}{1+f(k)\left(\ww^{k+1}(S^0)+\ww^{k+1}(S^1)\right)}\right)\mu^{k+1}_{S^0}.
\]
As we just argued with $\mu^{k+1}_{S^0}$, whenever we add measure $v$ to $\mu^{k+1}_{S^1}$, we also add at least $v/(1+f(k)(\ww^{k+1}(S^0)+\ww^{k+1}(S^1)))$ to $\mu^k_{S}$. So
\[
\mu^k_S[\lambda] \geq \frac{\mu^{k+1}_{S^0}[\lambda]+\mu^{k+1}_{S^1}[\lambda]}{1+f(k)\left(\ww^{k+1}(S^0)+\ww^{k+1}(S^1)\right)} = \frac{\ww^{k+1}(S^0)+\ww^{k+1}(S^1)}{1+f(k)\left(\ww^{k+1}(S^0)+\ww^{k+1}(S^1)\right)}.
\]
Putting it all together,
\begin{align*}
\Pf[k]\mu^k_S(X) &= f(k)\mu^k_S[\lambda] + \Pf[k+1]\nu_0(Y) \\
	&\geq f(k)\left(\frac{\ww^{k+1}(S^0)+\ww^{k+1}(S^1)}{1+f(k)\left(\ww^{k+1}(S^0)+\ww^{k+1}(S^1)\right)}\right) \\
	&\hspace{0.5in} + \left(\frac{1}{1+f(k)\left(\ww^{k+1}(S^0)+\ww^{k+1}(S^1)\right)}\right)\Pf[k+1]\mu^{k+1}_{S^0}(Y) \geq 1. \qedhere
\end{align*}
\end{proof}

The fact that we can build left-c.e.\ measures ``almost realizing'' the $f$-capacity of $\Sigma^0_1$-classes allows us to connect $\Cf$-randomness to the behavior of $\KM$.

\begin{lem}\label{lem:divergence}
Assume that $f\colon\omega\to[0,\infty)$ is amicable and $|f|$ is computable. If $X\in 2^\omega$ is not $\Cf$-random, then $\displaystyle\sum_{n\in\omega} f(n)2^{-\KM(X\uh n)} = \infty$.
\end{lem}
\begin{proof}
Let $\{U_m\}_{m\in\omega}$ be a $\Cf$-test covering $X$. We may assume that we have an effective sequence of good enumerations of c.e.\ sets generating the $\Sigma^0_1$-classes $U_m$. From these good enumerations, we get an effective sequence $\{\mu_m\}$ of left-c.e.\ measures such that $(\forall X\in U_m)\; \Pf\mu_m(X)\geq 1$ and $\mu_m[\lambda] \leq a\Cf[][U_m]\leq a2^{-m}$, where $a=2\|f\|+2$. Consider the left-c.e.\ measure $\mu = \sum_{m\in\omega} \mu_m$. This is a finite measure because $\mu[\lambda] = \sum_{m\in\omega} \mu_m[\lambda] \leq \sum_{m\in\omega} a2^{-m} = 2a$. Therefore, there is a constant $c$ such that $(\forall\sigma)\; \KM(\sigma)\leq -\log(\mu[\sigma]) + c$. Rearranging, $(\forall\sigma)\; 2^{-\KM(\sigma)}\geq 2^{-c}\mu[\sigma]$. So we have
\begin{align*}
\sum_{n\in\omega} f(n)2^{-\KM(X\uh n)} &\geq 2^{-c}\sum_{n\in\omega} f(n)\mu[X\uh n] \\
	&= 2^{-c}\sum_{n\in\omega}f(n)\sum_{m\in\omega}\mu_m[X\uh n] = 2^{-c}\sum_{m\in\omega}\sum_{n\in\omega} f(n)\mu_m[X\uh n] \\
	&= 2^{-c}\sum_{m\in\omega}\Pf\mu_m(X) \geq 2^{-c}\sum_{m\in\omega} 1 = \infty. \qedhere
\end{align*}
\end{proof}

%%%%%%%%
%%%%%%%%
\section{Characterizing \texorpdfstring{$f$}{f}-energy randomness}
%%%%%%%%
%%%%%%%%

We get our main result by combining Lemma~\ref{lem:divergence} with Corollary~\ref{cor:one-direction} and the second author's characterization of $f$-energy randomness as $\Cf$-randomness (see~\cite{Rute:aa} or Theorem~\ref{thm:rute} above).

\begin{thm}\label{thm:main}
Assume that $f\colon\omega\to[0,\infty)$ is amicable and $|f|=\sum_{n\in\omega} f(n)2^{-n}$ is computable. Then $X$ is $f$-energy random if and only if $\displaystyle\sum_{n\in\omega} f(n)2^{-\KM(X\uh n)} < \infty$.
\end{thm}
\begin{proof}
Implicit in our assumption that $|f|$ is computable is the assumption that $|f|<\infty$. Of course, by Remark~\ref{rem:too-fast}, the conclusion of the theorem holds trivially if $|f|=\infty$. By Remark~\ref{rem:fin-support}, we may also assume that $f$ has infinite support. Therefore, we may use the work of Sections~\ref{sec:capacity} and~\ref{sec:effectiveness}.

By Corollary~\ref{cor:one-direction}, if $X$ is $f$-energy random, then $\sum_{n\in\omega} f(n)2^{-\KM(X\uh n)} < \infty$. Now assume that $X$ is not $f$-energy random. Then by Theorem~\ref{thm:rute}, it is not $\Cf$-random. So by Lemma~\ref{lem:divergence}, we have $\sum_{n\in\omega} f(n)2^{-\KM(X\uh n)} = \infty$.
\end{proof}

%%%%%%%%
%%%%%%%%
\section{Applications}\label{sec:applications}
%%%%%%%%
%%%%%%%%

%%%%%%%%
\subsection*{\texorpdfstring{\boldmath$s$}{s}-energy randomness}
%%%%%%%%

In the introduction and more thoroughly in Section~\ref{subsec:s-energy}, we discussed the notion of $s$-energy randomness as introduced by Diamondstone and Kjos-Hanssen \cite{DK:12}. We noted that $s$-energy randomness is equivalent to $f$-energy randomness for $f(n)=2^{sn}$. Now fix a computable real number $s\in [0,1)$ and let $f(n)=2^{sn}$. Of course, $|f|=1/(1-2^{s-1})$ is computable. Furthermore, $f$ in nonzero and nondecreasing, hence it is amicable. So by the work of the previous section, $X\in 2^\omega$ is $s$-energy random if and only if
\[
\sum_{n\in\omega} 2^{sn-\KM(X\uh n)} < \infty.
\]
This proves Theorem~\ref{thm:main-s-energy}. It is natural to ask if this result holds for noncomputable dimensions $s$; our proofs do not appear to generalize.  
%\hl{It seems clear that the proof generalizes for noncomputable $s$ to give that $X$ is $s$-energy random iff $\sum_{n\in\omega} 2^{sn-\KM^s(X\uh n)} < \infty$.  \textbf{Right????}  However, this is different than $\sum_{n\in\omega} 2^{sn-\KM(X\uh n)} < \infty$.  To see this, let $X$ be $1/2$-energy random and let $s < 1/2$ be Turing equivalent to $X$.  Then $\sum_{n\in\omega} 2^{sn-\KM(X\uh n)} < \infty$, but $\KM^s(X\uh n)$ is bounded above by a constant, so $\sum_{n\in\omega} 2^{sn-\KM^s(X\uh n)} = \infty$.}

%%%%%%%%
\subsection*{Extensions to several symbols}
%%%%%%%%

The definitions and results of this paper naturally extended to alphabets with more than two symbols as follows.  Let $A$ be a finite alphabet.

\begin{defn}   A sequence $X \in A^\omega$ is \emph{$f$-energy random} if $X$ is random for some probability measure $\mu$ on $A^\omega$ with finite $f$-energy,
\[
\int_{A^\omega} \sum_{n\in \omega} f(n) \mu[X\uh n]\,d\mu(X).
\]
\end{defn}

\noindent The definitions of $f$-potential and $f$-capacity remain basically the same, namely
\begin{gather*}
\Pf \mu (X) = \sum_{n\in \omega} f(n) \mu[X\uh n], \\
\text{and }\Cf(U) = \inf\;\{\mu(A^\omega)\colon \mu\in\+M\text{ and }(\forall X\in U)\; \Pf\mu(X)\geq 1\}.
\end{gather*}
Moreover,  $|f| = \sum_{n\in\omega} f(n) |A|^n$.  As before $f\colon\omega\to[0,\infty)$ is \emph{amicable} if 
\[
\displaystyle\sup_{k\in\omega} f(k)\Cf[k+1](A^\omega)<\infty.
\]
The universal left-c.e.\ semimeasure $\mathbf{M}$ is defined as usual on $A^{<\omega}$ and a priori complexity is still defined as $KM(\sigma) = -\log_2 \mathbf{M}(\sigma)$ for $\sigma \in A^{<\omega}$. The proof of Theorem~\ref{thm:main} naturally extends to several symbols to give the following result.

\begin{thm}\label{thm:main2}
Assume that $f\colon\omega\to[0,\infty)$ is amicable and $|f|=\sum_{n\in\omega} f(n)|A|^{-n}$ is computable. Then $X \in A^\omega$ is $f$-energy random if and only if
\[
\sum_{n\in\omega} f(n) 2^{-\KM(X\uh n)} < \infty.
\]
\end{thm}

%%%%%%%%
\subsection*{s-randomness for several symbols}
%%%%%%%%

Let $A$ be a finite alphabet. Let $d = \log_2 |A|$ be the ``dimension'' of $A^\omega$.  For $X,Y\in A^\omega$, continue to use the metric $\rho(X,Y) = \inf\{2^{-n} : X\uh n = Y \uh n\}$.

\begin{defn}   Fix computable $s>0$. A sequence $X \in A^\omega$ is \emph{$s$-energy random} if $X$ is random for some measure $\mu$ on $A^\omega$ with finite Riesz $s$-energy,
\[ 
\RE(\mu) = \int_{A^\omega} \int_{A^\omega} \rho(X,Y)^{-s} \;d\mu(Y)d\mu(X).
\]
Fix a positive integer $k$. A sequence $X \in A^\omega$ is \emph{$\log^k$-energy random} if $X$ is random for some measure $\mu$ on $A^\omega$ with finite Riesz $\log^k$-energy,
\[ 
\RElog(\mu) = \int_{A^\omega} \int_{A^\omega} (-\log_2 \rho(X,Y))^{k} \;d\mu(Y)d\mu(X).
\]
A sequence $X \in A^\omega$ is \emph{$0$-energy random} if and only if it is $\log^1$-energy random.
\end{defn}

By the same proof as Corollary~\ref{cor:s-energy-f-energy}, $X \in A^\omega$ is $s$-energy random if and only if it is $f$-energy random for $f(n)=2^{sn}$.  Further, generalizing Corollary~\ref{cor:log-energy-equiv}, $X \in A^\omega$ is $\log^k$-energy random if and only if it is $f$-energy random for $f(n)=n^{k-1}$.  There are no $s$-energy randoms for $s\geq d$.

By Theorem~\ref{thm:main2}, $X\in A^\omega$ is $s$-energy random (for computable $s$) if and only if
\[
\sum_{n\in\omega} 2^{sn-\KM(X\uh n)} < \infty,
\]
and $X\in A^\omega$ is $\log^k$-energy random if and only if
\[
\sum_{n\in\omega} n^{k-1} 2^{-\KM(X\uh n)} < \infty.
\]

\begin{rem}
An important reason to consider larger alphabets is the study of energy randomness on $\R^d$.  While the focus of this paper is Cantor space, much of potential theory and its applications occur on $\R^d$.  A real $x\in\R$ can be represented via its binary expansion $X\in2^\omega$.  Similarly, a vector $(x_1,x_2)\in\R^2$ can be represented by a sequence $X\in4^\omega$ (interleave the binary expansions of $x_1$ and $x_2$ and take the resulting sequences of pairs $\{00,01,10,11\}$).  Using $2^d$ symbols is the correct way to represent $\R^d$ with strings of symbols since this representation preserves the potential-theoretic properties of energy and capacity (up to a constant multiple, see \cite[Thm~3.1]{Pemantle:1995aa}).  More details about energy randomness on $\R^d$ and its applications to multidimensional Martin-L\"of Brownian motion will be given in Rute \cite{Rute:aa}.
\end{rem}

%
% References
%

\bibliographystyle{plain}
\bibliography{references}

\begin{thebibliography}{10}

\bibitem{Allen:aa}
Kelty Allen, Laurent Bienvenu, and Theodore~A. Slaman.
\newblock On zeros of {M}artin-{L}\"of random {B}rownian motion.
\newblock {\em J. Log. Anal.}, 6(9):1--34, 2014.

\bibitem{Asarin.Pokrovskii:1986}
E.~A. Asarin and A.~V. Pokrovski{\u\i}.
\newblock Application of {K}olmogorov complexity to the analysis of the
  dynamics of controllable systems.
\newblock {\em Avtomat. i Telemekh.}, (1):25--33, 1986.

\bibitem{Barmpalias2007}
George Barmpalias, Paul Brodhead, Douglas Cenzer, Seyyed Dashti, and Rebecca
  Weber.
\newblock Algorithmic randomness of closed sets.
\newblock {\em J. Logic Comput.}, 17(6):1041--1062, 2007.

\bibitem{Day:2013aa}
Adam~R. Day and Joseph~S. Miller.
\newblock Randomness for non-computable measures.
\newblock {\em Trans. Amer. Math. Soc.}, 365(7):3575--3591, 2013.

\bibitem{DK:12}
David Diamondstone and Bj{\o}rn Kjos-Hanssen.
\newblock Martin-{L}\"of randomness and {G}alton-{W}atson processes.
\newblock {\em Ann. Pure Appl. Logic}, 163(5):519--529, 2012.

\bibitem{Fouche2000}
Willem~L. Fouch{\'e}.
\newblock The descriptive complexity of {B}rownian motion.
\newblock {\em Adv. Math.}, 155(2):317--343, 2000.

\bibitem{G:80}
P{\'e}ter G{\'a}cs.
\newblock Exact expressions for some randomness tests.
\newblock {\em Z. Math. Logik Grundlag. Math.}, 26(5):385--394, 1980.

\bibitem{Lyons:aa}
Russell Lyons and Yuval Peres.
\newblock {\em Probability on Trees and Networks}.
\newblock Cambridge University Press.
\newblock In preparation. Current version available at
  \url{http://pages.iu.edu/\string~rdlyons/}.

\bibitem{Miller.Yu:ta1}
Joseph~S. Miller and Liang Yu.
\newblock On initial segment complexity and degrees of randomness.
\newblock {\em Trans. Amer. Math. Soc.}, 360:3193--3210, 2008.

\bibitem{Pemantle:1995aa}
Robin Pemantle and Yuval Peres.
\newblock Galton-{W}atson trees with the same mean have the same polar sets.
\newblock {\em Ann. Probab.}, 23(3):1102--1124, 1995.

\bibitem{Reimann:2008vn}
Jan Reimann.
\newblock Effectively closed sets of measures and randomness.
\newblock {\em Ann. Pure Appl. Logic}, 156(1):170--182, 2008.

\bibitem{RS:15}
Jan Reimann and Theodore~A. Slaman.
\newblock Measures and their random reals.
\newblock {\em Trans. Amer. Math. Soc.}, 367(7):5081--5097, 2015.

\bibitem{Rute:aa}
Jason Rute.
\newblock Algorithmic randomness for capacities with applications.
\newblock In preparation.

\end{thebibliography}

\end{document}